\documentclass{article}

\usepackage{amsfonts,amsmath,amssymb,amsthm}
\usepackage{mathscinet}
\usepackage{mathrsfs}

\usepackage{enumitem}
\setitemize{itemsep=0em}

\usepackage{color}
\usepackage[breaklinks]{hyperref}
\definecolor{tocolor}{rgb}{.1,.1,.1}
\definecolor{urlcolor}{rgb}{.2,.2,.6}
\definecolor{linkcolor}{rgb}{.1,.1,.5}
\definecolor{citecolor}{rgb}{.4,.2,.1}
\definecolor{gray}{rgb}{.8,.8,.8}
\hypersetup{
	backref=true,
	colorlinks=true,
	urlcolor=urlcolor,
	linkcolor=linkcolor,
	citecolor=citecolor,
	pdfauthor = {Brent Pym and Travis Schedler},
	pdftitle = {Holonomic Poisson manifolds and deformations of elliptic algebras}
	}

\usepackage{aliascnt}

\newcommand{\thdef}[2]{
	\newaliascnt{#1}{theorem}  
	\newtheorem{#1}[#1]{#2}
	\aliascntresetthe{#1}  
	\newtheorem*{#1*}{#2}
	\expandafter\newcommand\expandafter{\csname #1autorefname\endcsname}{#2}
}

\newtheorem{theorem}{Theorem}[section]
\newtheorem*{theorem*}{Theorem}

\thdef{lemma}{Lemma}
\thdef{corollary}{Corollary}
\thdef{proposition}{Proposition}
\theoremstyle{definition}
\thdef{definition}{Definition}

\theoremstyle{remark}
\thdef{remark}{Remark}
\thdef{example}{Example}

\newcommand{\defn}[1]{\textbf{\emph{#1}}}

\usepackage{xypic}
\usepackage{tikz,tikz-cd}

\newcommand{\abrac}[1]{\left\langle#1\right\rangle}
\newcommand{\cbrac}[1]{\left\{#1\right\}}
\newcommand{\rbrac}[1]{\left(#1\right)}
\newcommand{\set}[2]{\left\{#1 \,\middle|\, #2 \right\}}

\newcommand{\CC}{\mathbb{C}}
\newcommand{\PP}{\mathbb{P}}

\newcommand{\X}{\mathsf{X}}
\newcommand{\U}{\mathsf{U}}
\newcommand{\E}{\mathsf{E}}
\newcommand{\Y}{\mathsf{Y}}
\newcommand{\bL}{\mathsf{L}}
\newcommand{\W}{\mathsf{W}}
\newcommand{\bS}{\mathsf{S}}
\newcommand{\D}{\mathsf{D}}

\newcommand{\Char}[1]{\mathsf{Char}(#1)}
\newcommand{\ctb}[1]{\mathsf{T}^*#1}

\newcommand{\reg}{\mathrm{reg}}
\newcommand{\sing}{\mathrm{sing}}
\newcommand{\red}{\mathrm{red}}
\DeclareMathOperator{\codim}{codim}
\DeclareMathOperator{\coker}{coker}

\newcommand{\pis}{\pi^\sharp}

\newcommand{\cT}[1]{\mathcal{T}_{#1}} 
\newcommand{\Tcx}[2][\bullet]{\mathcal{T}_{#2}^{#1}} 
\newcommand{\Tpol}[2][\bullet]{\wedge^{#1} \cT{#2}} 
\newcommand{\forms}[2][\bullet]{\Omega^{#1}_{#2}} 
\newcommand{\cO}[1]{\mathcal{O}_{#1}} 
\newcommand{\cI}{\mathcal{I}} 
\newcommand{\m}{\mathfrak{m}} 
\newcommand{\sD}[2][]{\mathcal{D}^{#1}_{#2}} 
\newcommand{\sH}{\mathcal{H}} 
\newcommand{\sN}{\mathcal{N}} 
\newcommand{\cM}[1][\bullet]{\mathcal{M}^{#1}} 
\newcommand{\cMpi}[1][\bullet]{\cM[#1]_\pi} 
\newcommand{\sQ}{\mathcal{Q}}
\newcommand{\sF}{\mathcal{F}}
\DeclareMathOperator{\Sym}{Sym}

\newcommand{\Pic}[2][]{\mathsf{Pic}_{#1}(#2)}
\newcommand{\coH}[2][\bullet]{\mathsf{H}^{#1}(#2)}
\newcommand{\IH}[2][\bullet]{\mathsf{IH}^{#1}(#2)}
\newcommand{\scoH}[2][\bullet]{\mathcal{H}^{#1}(#2)}
\newcommand{\Hpi}[2][\bullet]{\mathsf{H}^{#1}_\pi(#2)} 
\newcommand{\sHpi}[1][\bullet]{\mathcal{H}^{#1}_\pi} 
\newcommand{\sIC}[1]{\mathcal{IC}_{#1}} 
\newcommand{\sK}{\mathcal{K}}
\newcommand{\sExt}[2][\bullet]{\mathcal{E}xt^{#1}(#2)}
\newcommand{\sHom}[1]{\mathcal{H}om(#1)}

\newcommand{\dd}{\mathrm{d}}
\newcommand{\dpi}{\dd_\pi}
\newcommand{\lie}[1]{\mathscr{L}_{#1}}
\newcommand{\hook}[1]{\iota_{#1}}
\newcommand{\cvf}[1]{\partial_{#1}}
\newcommand{\BV}[1]{\Delta_{#1}}
\newcommand{\Ham}[1]{\xi_{#1}}
\newcommand{\Mod}{\zeta}

\theoremstyle{theorem}
\newtheorem*{thm-mod}{Theorem \ref{thm:modular}}
\newcommand{\thmmod}{
Suppose that $(\X,\pi)$ is a complex Poisson manifold.  Then the following statements hold:
\begin{enumerate}
\item $\Char{\X,\pi}$ is contained in the union of the conormal bundles of the modular leaves.  In particular, if every point has a neighbourhood
in which the modular foliation has finitely many leaves, then $(\X,\pi)$ is holonomic.
\item The conormal bundle of every even-dimensional modular leaf is contained in  $\Char{\X,\pi}$.  Hence, if $(\X,\pi)$ is holonomic, there are at most finitely many such leaves in any compact subset.
\end{enumerate}
}

\newtheorem*{thm-norm}{Theorem \ref{thm:normal}}
\newcommand{\thmnormal}{
Suppose that $(\X,\pi)$ is a holonomic Poisson manifold whose degeneracy divisor $\Y$ is normal, let $\U = \X \setminus \Y$ be the open symplectic leaf, and let $\Hpi{\X}$ be the hypercohomology of the Lichnerowicz complex.  Then the natural restriction map
\[
\Hpi[i]{\X} \to \coH[i]{\U;\CC}
\]
is an isomorphism for $i=0,1$ and injective for $i=2$.
}

\newtheorem*{thm-ell}{Theorem \ref{thm:elliptic}}
\newcommand{\thmelliptic}{
Let $\pi = q_{d,1}(\E,\zeta)$ be a Feigin--Odesskii Poisson structure on $\PP^{d-1}$, where  $d$ is odd.  Then $\pi$ is holonomic, and we have canonical isomorphisms
\begin{align*}
\Hpi[0]{\PP^{d-1}} &= \CC & \Hpi[1]{\PP^{d-1}}&=0 & \Hpi[2]{\PP^{d-1}} &\cong \coH[1]{\E;\CC}.
\end{align*}
Moreover, variation of the parameters $(\E,\Mod)$ gives the universal analytic deformation of $\pi$.
}

\begin{document}

\title{Holonomic Poisson manifolds and deformations of elliptic algebras}

\author{Brent Pym\thanks{School of Mathematics, University of Edinburgh, \href{mailto:brent.pym@ed.ac.uk}{brent.pym@ed.ac.uk}} \and Travis Schedler\thanks{Department of Mathematics, Imperial College London, \href{mailto:t.schedler@imperial.ac.uk}{t.schedler@imperial.ac.uk}}}
\maketitle

\begin{abstract}
We introduce a natural nondegeneracy condition for Poisson structures, called holonomicity, which is closely related to the notion of a log symplectic form.  Holonomic Poisson manifolds are privileged by the fact that their deformation spaces are as finite-dimensional as one could ever hope: the corresponding derived deformation complex is a perverse sheaf.  We develop some basic structural features of these manifolds, highlighting the role played by the divergence of Hamiltonian vector fields.  As an application, we establish the deformation-invariance of certain families of Poisson manifolds defined by Feigin and Odesskii, along with the ``elliptic algebras'' that quantize them.
\end{abstract}

\tableofcontents

\newpage

\section{Introduction}

During the past decade, there has been renewed interest in the deformations of complex analytic Poisson manifolds, owing to their importance in the study of moduli spaces, noncommutative algebras, and generalized complex structures.  Much of the work has focused on the question of unobstructedness of deformations, generalizing the classical result of Bogomolov~\cite{Bogomolov1978} in the symplectic case.  For example, Ran~\cite{Ran2017a,Ran2017b} extended Bogomolov's result to Poisson manifolds given by normal crossings compactifications of symplectic manifolds, under some assumptions on the behaviour of the Poisson structure at infinity.

In a different direction (via generalized complex geometry), there was a remarkable discovery of Goto~\cite{Goto2009}, generalized in turn by Hitchin~\cite{Hitchin2012}, and Fiorenza--Manetti~\cite{Fiorenza2012}: just as a  function on a Poisson manifold gives a symmetry via the flow of its Hamiltonian vector field, a closed $(1,1)$-form gives a \emph{deformation}.  Under quite general conditions (e.g.~if the $\partial\overline{\partial}$-lemma holds), this deformation is unobstructed. Hitchin~\cite{Hitchin2012} used this viewpoint to link Poisson deformations of the Hilbert scheme of the projective plane with the Hilbert schemes of noncommutative planes introduced by Nevins--Stafford~\cite{Nevins2007}.

In this paper, we ask a different question: which Poisson manifolds $\X$ have  finite-dimensional deformation spaces, and how can such deformation spaces be determined?   In the symplectic case, all deformations are locally trivial by the Darboux theorem, and the problem boils down to computing the second de Rham cohomology.  But in general, the problem is very difficult.  Instead of the de Rham complex, the deformations are governed by the polyvector fields $\Tpol{\X}$, with differential obtained from the Poisson structure, as explained by~Lichnerowicz~\cite{Lichnerowicz1977}.  The global hypercohomology of this complex (often called the ``Poisson cohomology'') controls the deformations of $\X$ not just as a holomorphic Poisson manifold, but as a generalized complex manifold~\cite{Gualtieri2011,Hitchin2003}.

The Lichnerowicz complex is extremely sensitive to the local singularities of the Poisson bracket.  In contrast with the symplectic case, the local deformations spaces are almost always infinite-dimensional, and highly obstructed.  It therefore seems natural to single out those Poisson manifolds whose local deformation spaces are finite-dimensional, even at the ``derived'' level.  In other words, we look for Poisson structures for which the stalks of all the cohomology sheaves of $\Tpol{\X}$ are finite-dimensional. 

The theory of algebraic analysis offers a powerful toolbox for addressing such questions. Since the Lichnerowicz differential is a differential operator,  it induces a complex $\Tpol{\X}\otimes\sD{\X}$ of modules over the sheaf of differential operators $\sD{\X}$.  The singular support of this complex gives a canonical coisotropic subvariety of the cotangent bundle $\ctb{\X}$, which we call the \defn{characteristic variety} of the Poisson manifold.  We introduce the following natural definition:

\begin{definition}
A Poisson manifold is \defn{holonomic} if its characteristic  variety is Lagrangian.
\end{definition}

By Kashiwara's constructibility theorem for D-modules~\cite{Kashiwara1974/75}, holonomicity ensures (and is roughly equivalent to) the
finite-dimensionality of the stalks of the Lichnerowicz cohomology
sheaves.

Holonomicity is a sort of nondegeneracy condition on a Poisson manifold $\X$, constraining the singularities that can occur along strata of arbitrary codimension.  For instance, in codimension zero, it is equivalent to the existence of an open dense symplectic leaf.  But the Poisson structure may then degenerate along a hypersurface $\Y\subset \X$ that is  foliated by symplectic leaves of smaller dimension.  In codimension one, holonomicity is equivalent to requiring that $\Y$ be reduced, i.e., all irreducible components occur with multiplicity one.  Equivalently, the Poisson structure must be defined by a ``log symplectic'' form in the sense of Goto~\cite{Goto2002}.  We note that log symplectic structures have recently received considerable attention starting with the works~\cite{Gualtieri2014,Guillemin2014,Radko2002} in the $C^\infty$ setting and \cite{Goto2002,Lima2014,Ran2017b} in the holomorphic setting, assuming that $\Y$ is smooth or has normal crossings singularities.

In higher codimension, ``holonomic'' is a stronger condition than ``log symplectic'', but it is still flexible enough to allow some quite complicated singularities.  For instance, any plane curve singularity can appear in codimension two, and simple elliptic surface singularities can appear in codimension three~\cite{Pym2017}. Thus, although holonomic Poisson manifolds are quite rigid, there are many natural and interesting examples arising from partial compactifications of holomorphic symplectic manifolds, such as various moduli spaces.

The development of the general theory of holonomic Poisson manifolds is a work in progress; we report here on some preliminary results that illustrate the utility of the approach in concrete applications.  We explain that holonomicity is closely related to an intriguing local symmetry of Poisson manifolds discovered by Brylinski--Zuckerman~\cite{Brylinski1999} and Weinstein~\cite{Weinstein1997}: the modular vector field.  It is the ``divergence'' of the Poisson tensor---the obstruction to the existence of Hamiltonian-invariant volume forms.  (It is also the semi-classical limit of the Nakayama automorphism of noncommutative geometry.)  By flowing the symplectic leaves of a Poisson manifold along its modular vector field, we obtain a new foliation whose leaves may be strictly larger.  We prove the following result, giving effective geometric criteria for (non-)holonomicity, analogous to results in \cite{Etingof2016} concerning coinvariants for Lie algebras of vector fields:
\begin{thm-mod}
\thmmod
\end{thm-mod}

In \autoref{sec:compfact}, we describe a key feature of holonomic Poisson manifolds: their Lichnerowicz complexes are perverse sheaves (\autoref{thm:perverse}), a much stronger condition than finite-dimensionality of the stalk cohomology.  This property allows one to bring the powerful tools of intersection cohomology theory to bear on the deformation problem, giving systematic control over the contributions made by singularities of high codimension, which tend to be difficult to access by direct means.  A typical output of this approach is the following result:
\begin{thm-norm}
\thmnormal
\end{thm-norm}
The practical calculation of the cohomology is then facilitated by a Gysin sequence that relates the cohomology of $\U$ to the \emph{intersection cohomology} of $\Y$.  In higher degree, or in the non-normal case, the link with the cohomology of the complement becomes much more complicated, involving additional contributions from intersection cohomology of local systems on various singular strata. (See \autoref{sec:surf} for the case of surfaces.)  We hope to describe these ``composition factors'' in more detail in  future work.

As an application, we consider the ``elliptic'' Poisson brackets constructed by Feigin and Odesskii~\cite{Feigin1989,Feigin1998} in the late 1980s. They give a family  $q_{d,r}(\E,\Mod)$ of Poisson brackets on the projective space $\PP^{d-1}$, determined up to isomorphism by a pair $(d,r)$ of coprime integers, a smooth curve $\E$ of genus one, and a vector field $\Mod$ on $\E$.  They arise most naturally by viewing $\PP^{d-1}$ as a certain moduli space of bundles on $\E$; see \cite{Polishchuk1997}.  One might ask whether the two continuous parameters $(\E,\Mod)$ describe the full deformation space of these Poisson structures.  For $r =1$, and small values of $d$, this follows from various direct classification results~\cite{Cerveau1996,Loray2013,Polishchuk1997,Pym2017}, but those methods do not extend easily to arbitrary dimension.  Our results above apply for $r=1$ and arbitrary odd $d$, allowing us to establish the deformation-invariance of the family in that case:
\begin{thm-ell}
\thmelliptic
\end{thm-ell}
Since the family $q_{d,1}(-,-)$ is algebraic, we have an immediate corollary:
\begin{corollary}\label{cor-ell}
If $d$ is odd, the family  $q_{d,1}(-,-)$ forms a Zariski open set in the moduli space of Poisson structures on $\PP^{d-1}$.  
\end{corollary}
Note that Feigin--Odesskii~\cite{Feigin1989,Feigin1998} also introduced noncommutative ``elliptic algebras'' $Q_{d,r}(-,-)$ that quantize their Poisson structures.  Tate and Van den Bergh~\cite{Tate1996} showed that these algebras are Artin--Schelter regular with Hilbert series $(1-t)^{-d}$. Hence by \cite[Theorem 2.10]{Pym2015}, based on results in deformation quantization~\cite{Bondal1993a,Dolgushev2009,Kontsevich2001,Kontsevich2003}, we have a further immediate consequence:
\begin{corollary}
If $d$ is odd, the family $Q_{d,1}(-,-)$ forms a Zariski open set in the moduli space of quadratic Calabi--Yau algebras with Hilbert series $(1-t)^{-d}$.
\end{corollary}

In closing, let us remark that our approach is inspired by a recent
series of papers by Etingof and the second author, which used similar
techniques to address a different problem in Poisson geometry (see,
e.g., \cite{Etingof2010,Etingof2014,Etingof2017}).  Using a single
D-module, they calculated the space of Poisson traces for various
singular symplectic varieties arising in representation theory. They
showed that the holonomicity of their D-module is equivalent to the
local finiteness of the foliation by symplectic leaves, thus
establishing finite-dimensionality of the space of Poisson traces in
that case.  In the context of symplectic linear quotients,
finite-dimensionality of the space of Poisson traces had been
conjectured by Alev--Farkas~\cite{Alev2003} and proved by
Berest--Etingof--Ginzburg~\cite[Appendix]{Berest2004}, and their proof
inspired the D-module constructions of \emph{op.~cit.}

We note that the D-module controlling Poisson traces is different from
the complex used here; it is related to the Poisson homology of
Brylinski~\cite{Brylinski1988} rather than the Poisson cohomology of
Lichnerowicz, and the two are not, in general, dual.  Moreover, for
smooth varieties, the local finiteness of the symplectic foliation is
equivalent to the variety being symplectic, which is much
stronger (and less subtle) than the holonomicity condition considered
in the present paper.  In fact, it is elementary to see that
Brylinski's complex has finite-dimensional stalk cohomology if and only if the
manifold is symplectic.  Thus the question of
finite-dimensionality for various cohomological invariants of Poisson
manifolds is sensitive to the specific invariant under consideration.

\paragraph{Acknowledgements:} It is an honour to dedicate this paper to Nigel Hitchin on the occasion of his 70th birthday, with gratitude.  At various stages of this project, B.P.~has been supported by EPSRC Grant EP/K033654/1, a Junior Research Fellowship at Jesus College (Oxford) and a William Gordon Seggie Brown Research Fellowship at the University of Edinburgh.  T.S.~was partially supported by NSF Grant DMS-1406553.  This research was supported by an EPSRC Platform Grant at Imperial College London.


\section{Holonomic Poisson manifolds}

\subsection{Poisson structures and the Lichnerowicz complex}

In this section, we recall some basic definitions in Poisson geometry (see, e.g.~\cite{Dufour2005,Laurent-Gengoux2013,Polishchuk1997} for details).  Let $\X$ be a complex manifold.  We denote by $\cO{\X}$, $\cT{\X}$ and $\forms{\X}$ the sheaves of  holomorphic functions, vector fields and differential forms, respectively.  Recall that a \defn{Poisson structure on $\X$} is a Lie bracket
\[
\{-,-\} : \cO{\X} \times \cO{\X} \to \cO{\X}
\]
that is a derivation in each argument.  Thus a Poisson structure determines, and is  determined by, a global holomorphic bivector field
\[
\pi \in \Gamma(\X,\Tpol[2]{\X})
\]
according to the formula $\{f,g\} = \abrac{\pi,\dd f \wedge \dd g}$ for $f,g \in \cO{\X}$.  The Jacobi identity for the bracket corresponds to the nonlinear integrability condition
\begin{align}
[\pi,\pi] = 0 \in \Gamma(\X,\Tpol[3]{\X}). \label{eqn:integrability}
\end{align}
where $[-,-]$ denotes the Schouten--Nijenhuis bracket of polyvector fields.

There is a natural map $\pis : \forms[1]{\X} \to \cT{\X}$ given by contraction with $\pi$.  Any function $f \in \cO{\X}$ has a Hamiltonian vector field $\Ham{f} = \pis(\dd f)$, defined so that $\lie{\Ham{f}}g = \{f,g\}$.  The Hamiltonian vector fields generate an involutive subsheaf of $\cT{\X}$, giving a possibly singular foliation by even-dimensional \defn{symplectic leaves}.

\begin{example}\label{ex:surf-intro}Throughout the paper, we will illustrate the general theory using Poisson surfaces (the case $\dim \X = 2$) as examples.  In this case, the condition \eqref{eqn:integrability} holds automatically, and therefore a Poisson structure $\pi \in \Tpol[2]{\X} \cong \det \cT{\X}$ is simply a section of the anti-canonical line bundle.  The zero locus of $\pi$, if nonempty, is a curve $\Y \subset \X$, giving an effective anticanonical divisor.  In local coordinates $(w,z)$ on $\X$, we may write
\[
\pi = f \cvf{w}\wedge\cvf{z}
\]
for some holomorphic function $f = f(w,z)$, so that the Hamiltonian vector fields of $w$ and $z$ are $\Ham{w} = f\cvf{z}$ and $\Ham{z} = -f\cvf{w}$.  They are linearly independent on $\U = \X \setminus \Y$ and vanish at the individual points of $\Y$.
As a result, the symplectic leaves come in two types: the open set $\U$ is a two-dimensional leaf, and the individual points of $\Y$ are zero-dimensional leaves. \qed
\end{example}

As observed by Lichnerowicz~\cite{Lichnerowicz1977}, the adjoint action of $\pi$ defines a differential
\[
\dpi = [\pi,-] : \Tpol{\X} \to \Tpol[\bullet +1]{\X}.
\]
This differential and the Schouten--Nijenhuis bracket make $\Tpol{\X}$ into a sheaf of differential graded Lie algebras, up to a shift in degree.  The cohomology sheaves $\sHpi = \scoH{\Tpol{\X},\dpi}$ have the following interpretations in low degrees: 
\begin{itemize}
\item $\sHpi[0] \subset \cO{\X}$ is the sheaf of Casimir functions---functions whose Hamiltonian vector fields vanish.
\item $\sHpi[1]$ is the sheaf of local infinitesimal outer automorphisms---vector fields $\xi$ such that $\lie{\xi}\pi = 0$, modulo the Hamiltonian ones.
\item $\sHpi[2]$ gives the isomorphism classes of first-order local deformations of $\pi$.
\item $\sHpi[3]$ gives the obstructions to lifting local deformations to higher order.
\end{itemize}

The interpretation of the global hypercohomology $\Hpi{\X}$ is more subtle. Gualtieri~\cite{Gualtieri2011} showed that it governs deformations of $(\X,\pi)$ as a generalized complex manifold (up to Courant isomorphism), rather than a complex Poisson manifold (up to analytic isomorphism).  The former has the advantage of capturing the full Hochschild cohomology of $\X$ when $\pi = 0$.  The latter can be described using the truncated complex $\Tpol[\ge 1]{\X}$ as in~\cite{Ginzburg2004,Namikawa2008}.  When $h^1(\cO{\X}) = h^2(\cO{\X}) = 0$, these two types of deformation are equivalent, and one has the following result:
\begin{theorem}\label{thm:gc-def}
If $(\X,\pi)$ is a compact complex Poisson manifold such that $h^{1}(\cO{\X}) = h^{2}(\cO{\X}) = 0$, then $(\X,\pi)$ has a versal deformation whose base is an analytic germ $\W \subset (\Hpi[2]{\X},0)$.  If $\Hpi[1]{\X}=0$, this deformation is universal.
\end{theorem}

\subsection{D-modules and holonomicity}
\label{sec:holonomic}

We now explain how to reinterpret the Lichnerowicz complex in the language of D-modules.
Let $\sD{\X}$ be the sheaf of differential operators on $\cO{\X}$.  Recall that $\sD{\X}$ is a sheaf of noncommutative algebras on $\X$; it contains $\cO{\X}$ as the scalar operators, and $\cT{\X}$ as the derivations of functions.  More generally, $\sD{\X}$ carries a natural increasing filtration $F^\bullet \sD{\X}$ by the order of differential operators, and the associated graded sheaf is the algebra $\cO{\ctb{\X}} = \Sym{\cT{\X}}$ of polynomial functions on the cotangent bundle $\ctb{\X}$.  The isomorphism $F^k\sD{\X}/F^{k-1}\sD{\X} \cong \Sym^k{\cT{\X}}$ is obtained by taking principal symbols of $k$th order operators. 

If $(\X,\pi)$ is a Poisson manifold, then the Lichnerowicz  differential $\dpi$ is a first-order differential operator.  Hence, following~\cite{Saito1989}, we can encode it in a natural complex of right modules over $\sD{\X}$
\[
\cMpi = (\xymatrix{
\cMpi[0] \ar[r] & \cMpi[1] \ar[r] & \cdots \ar[r] & \cMpi[n] } ),
\]
where $n = \dim \X$, as follows.  The terms in the complex are given by
\[
\cMpi[k] = \cbrac{ \textrm{differential operators }\cO{\X} \to \Tpol[k]{\X} } 
\]
while the right $\sD{\X}$-module structure and the differential $\cMpi \to \cMpi[\bullet+1]$ are obtained using the obvious compositions:
\[
\begin{tikzcd}[cells={nodes={}}]
        \arrow[loop,out=200,in=160,distance=2em]{}{\dpi}  \Tpol{\X} 
        &  \arrow{l}[above]{\cMpi} \cO{\X} \arrow[loop,swap,out=-25,in=25,distance=2em]{}{\sD{\X}} 
    \end{tikzcd}
\]
Notice that there is a canonical isomorphism $\cMpi[k] \cong \Tpol[k]{\X} \otimes_{\cO{\X}} \sD{\X}$  of right $\sD{\X}$-modules, obtained by sending a tensor $\xi\otimes \phi \in \Tpol[k]{\X}\otimes\sD{\X}$ to the operator $\cO{\X} \to \Tpol[k]{\X}$ defined by $g \mapsto \phi(g)\cdot \xi$. Hence we can recover the Lichnerowicz complex from $\cMpi$ as the tensor product
\[
(\Tpol{\X},\dpi) \cong \cMpi \otimes_{\sD{\X}} \cO{\X}.
\]

Let us recall that one can associate to any coherent $\sD{\X}$-module $\sN$ a $\CC^*$-invariant coisotropic subvariety of the cotangent bundle
\[
\Char{\sN} \subset \ctb{\X},
\]
called its \defn{characteristic variety}.  It is obtained by putting a ``good'' filtration on $\sN$, so that the associated graded sheaf becomes a module over $\cO{\ctb{\X}}$, and in fact a coherent sheaf on $\ctb{\X}$.  The characteristic variety is then the support of this sheaf.  We recall that $\sN$ is \defn{holonomic} if $\Char{\sN}$ is Lagrangian.  By convention, the zero module $\sN=0$ is also holonomic.

On a Poisson manifold $(\X,\pi)$, we have, for each $k \ge 0$, the cohomology module $\scoH[k]{\cMpi}$ of the complex $\cMpi$.  It is a coherent right $\sD{\X}$-module and therefore has a characteristic variety $\Char{\scoH[k]{\cMpi}} \subset \ctb{\X}$. 
\begin{definition}
The \defn{characteristic variety of $(\X,\pi)$} is the coisotropic subvariety
\[
\Char{\X,\pi} = \bigcup_{k \ge 0} \Char{\scoH[k]{\cMpi}} \subset \ctb{\X}.
\]
If $\Char{\X,\pi}$ is Lagrangian, we say that $(\X,\pi)$ is \defn{holonomic}.
\end{definition}
Thus $(\X,\pi)$ is holonomic if and only if all cohomology modules are so.  In that case, it turns out that $\scoH[k]{\cMpi} = 0$ for $k \ne n$ (see \autoref{prop:acyclic}).

We will not give a complete geometric description of the characteristic variety $\Char{\X,\pi}$ in this paper.  But we will nevertheless construct enough equations for these varieties to give useful criteria for (non-)holonomicity.  The Hamiltonian vector fields provide a good starting point:
\begin{lemma}\label{lem:hams}
Let $(\X,\pi)$ be a Poisson manifold, let $f \in \cO{\X}$ be a function and view the Hamiltonian vector field $\Ham{f} \in \cT{\X}$ as a linear function on $\ctb{\X}$.  Then $\Ham{f}$ vanishes on $\Char{\X,\pi} \subset \ctb{\X}$.  Put differently, the characteristic variety is contained in the union of the conormal bundles of the symplectic leaves.
\end{lemma}

\begin{proof}
Let $\hook{\dd f} : \Tpol{\X} \to \Tpol[\bullet -1]{\X}$ by the operator on the Lichnerowicz complex given by contraction with the differential $\dd f$.
We have the homotopy formula $\lie{\Ham{f}} = \dpi \hook{\dd f} + \hook{\dd f} \dpi$ as operators on $\Tpol{\X}$. 

Consider the natural filtration $F^\bullet \cMpi$ by the order of differential operators.  If $P \in F^j\cMpi$ is a cocycle that is an operator of order $j$, and we view the Hamiltonian vector field $\Ham{f} \in \cT{\X}$ as an element of $\sD{\X}$, we have the right action
\begin{align*}
P \cdot \Ham{f} &= P \lie{\Ham{f}} \\
&\equiv \lie{\Ham{f}} P \mod F^{j}\cMpi \\
&\equiv \dpi \hook{df} P \mod F^{j}\cMpi
\end{align*}
It follows that right multiplication by $\Ham{f}$ acts trivially on the associated graded of the cohomology $\scoH{\cMpi}$.  Hence $\Ham{f}$ vanishes on the characteristic variety.

The statement regarding conormal bundles is now immediate: since the span of the Hamiltonian vector fields in the tangent space is exactly the tangent space to the symplectic leaves, a point of $\ctb{\X}$ is annihilated by all Hamiltonians if and only if it lies in the conormal bundle of a symplectic leaf.
\end{proof}

\subsection{Generic nondegeneracy}
\label{sec:gen-symp}

 \autoref{lem:hams} above has the following immediate corollary:
\begin{corollary}\label{cor:symp}
If $(\X,\pi)$ is symplectic (i.e.~$\pi \in \Tpol[2]{\X}$ is nondegenerate), then $\Char{\X,\pi} \subset \ctb{\X}$ is the zero section.  In particular, $(\X,\pi)$ is holonomic.
\end{corollary}

%

At the opposite extreme, note that if $\pi = 0$, then the differential on $\cMpi$ is identically zero, so that $\scoH[0]{\cMpi}=\sD{\X}$ and $\Char{\X,\pi} = \ctb{\X}$.  Thus the zero Poisson structure is holonomic if and only if $\dim \X = 0$.

In general, holonomicity of a Poisson manifold $(\X,\pi)$ is a sort of nondegeneracy condition on the local singularities of $\pi$.  To see this, recall Weinstein's splitting theorem \cite{Weinstein1983}, which states that the germ $(\X,\pi)_p$ at any point $p \in \X$ is isomorphic to a product of Poisson structures $(\X,\pi)_p = (\W,\pi|_\W)_{p} \times (\X',\pi')_{p'}$, where $\W$ is the symplectic leaf through $p$ and $\pi'$ is a Poisson structure that vanishes at the point $p'$.   The transverse germ $(\X',\pi')_p$ is unique up to isomorphism and called the \defn{transverse Poisson structure at $p$}.  Evidently, the Lichnerowicz complex of the germ $(\X,\pi)_p$ decomposes as a product of the Lichnerowicz complexes of $(\W,\pi|_\W)_p$ and $(\X',\pi')_{p'}$, which easily gives the following:
\begin{lemma}
A Poisson manifold $(\X,\pi)$ is holonomic in a neighbourhood of $p \in \X$ if and only the transverse Poisson structure at $p$ is holonomic.
\end{lemma}

\begin{corollary}
Every holonomic Poisson manifold has an open dense symplectic leaf.
\end{corollary}
\begin{proof}
Every Poisson manifold $(\X,\pi)$ has an open dense set on which it is regular, i.e.~where the symplectic leaves all have the same dimension.  On this open set, the transverse Poisson structure is identically zero, so holonomicity forces the transverse structure to be zero-dimensional.
\end{proof}


\section{Holonomicity and the modular vector field}

\subsection{The modular vector field}

The basic strategy to understand the structure of the complex $\cMpi$ is to look for natural differential operators on $\Tpol{\X}$ that act trivially on the cohomology, and use their symbols to construct equations for $\Char{\X,\pi}$.  The prototypical example is the Lie derivative along Hamiltonian vector fields used in \autoref{lem:hams}.  In this section, we explain the role of another key operator: the modular vector field~\cite{Brylinski1999,Polishchuk1997,Weinstein1997}.  Let us recall its definition and basic properties.

Suppose for the moment that $\X$ is a Calabi--Yau manifold of dimension $n$, i.e.~it admits a global holomorphic volume form, or equivalently, a top degree polyvector field $\mu \in \Gamma(\X,\Tpol[n]{\X})$ that is nonvanishing (a covolume).  Then $\mu$ gives rise to an isomorphism
\[
\Tpol{\X} \cong \forms[n-\bullet]{\X}.
\]
Transporting the exterior derivative by this isomorphism, we obtain an operator $\BV{}$ of degree $-1$ on $\Tpol{\X}$, called the Batalin--Vilkovisky (BV) operator.  It is a derivation of the Schouten bracket:
\begin{align}
\BV{} [\alpha,\beta] = [\BV{}\alpha,\beta] - (-1)^{|\alpha|}[\alpha,\BV{}\beta] \label{eqn:bv-der}
\end{align}
where $|\alpha|$ denotes the degree of a homogeneous element $\alpha$.

The BV operator is a sort of divergence for polyvectors.  In particular, if $\xi$ is a vector field, the function $\BV{} \xi$ is the usual divergence of $\xi$ with respect to $\mu$:
\begin{align}
\lie{\xi}\mu = -(\BV{}\xi) \cdot \mu, \label{eqn:div}
\end{align}
so that $\BV{} \xi=0$ if and only if $\mu$ is invariant under the flow of $\xi$.

Applying the BV operator to a Poisson structure $\pi$, we obtain a vector field
\[
\Mod = \BV{}\pi \in \cT{\X},
\]
called the \defn{modular vector field}.  For a function $f \in \cO{\X}$ and its Hamiltonian vector field $\Ham{f}$, the derivation rule \eqref{eqn:bv-der} gives the identity
\begin{align}
\Mod(f) = -\BV{}\Ham{f} \label{eqn:div-ham}
\end{align}
so that  $\Mod = 0$ if and only if $\mu$ is invariant under all Hamiltonian flows.  One can also show that $\lie{\Mod}\pi = 0$, so that $\Mod$ is an infinitesimal symmetry of $\pi$. 

\begin{example}\label{ex:surf-CY}
Consider the Calabi--Yau surface $\X = \CC^2$  with the standard covolume form $\mu = \cvf{w}\wedge\cvf{z}$.  Suppose $\pi = f \cvf{w}\wedge\cvf{z}$ is a Poisson structure.  One readily computes that the modular vector field is given by
\[
\Mod = \BV{} \pi = (\cvf{w}f)\cvf{z}-(\cvf{z}f)\cvf{w}.
\]
Since $\Mod$ is a symmetry, it must be tangent to the zero locus $\Y \subset \X$ of $\pi$, i.e.~the locus where $f$ vanishes.   Moreover, since $\Mod$ vanishes exactly at the critical points of $f$, the zeros of the restriction $\Mod|_\Y$ are exactly the singular points of $\Y$.  Here ``singular point'' is taken in the scheme-theoretic sense: a singular point is any point where $f$ vanishes to order greater than one.  For example if $f = w^2$, then every point of the line $\{w=0\}$ is singular.  \qed
\end{example}

The modular vector field can be used to give a concrete presentation of the top cohomology of our complex of D-modules:
\begin{proposition}\label{prop:pres}
Let $\pi$ be a Poisson structure on a Calabi-Yau manifold, and let $\Mod = \BV{}\pi$ be its modular vector field. Then we have an isomorphism
\[
\scoH[n]{\cMpi} = \cMpi[n]/\dpi \cMpi[n-1] \cong \sD{\X} / \cI
\]
where $\cI\subset \sD{\X}$ is the right ideal generated by first-order differential operators of the form $\Mod(f) + \Ham{f} \in \cO{\X}\oplus\cT{\X}$ for $f \in\cO{\X}$.
\end{proposition}

\begin{proof}
The covolume $\mu$ gives an isomorphism $\Tpol[n]{\X} \cong \cO{\X}$, and hence $\cMpi[n] \cong \sD{\X}$ as right $\sD{\X}$-modules.  The isomorphism $\sD{\X} \to \cMpi[n]$ sends an element $\phi \in \sD{\X}$ to the operator $\cO{\X} \to \Tpol[n]{\X}$ defined by $g \mapsto \phi(g)\mu$.  We will show that this isomorphism identifies the submodule $\dpi \cMpi[n-1] \subset \cMpi[n]$  with the ideal $\cI\subset \sD{\X}$, by comparing their generators.

To this end, notice that the elements $\hook{\dd f}\mu$ for $f \in \cO{\X}$ generate $\Tpol[n-1]{\X}$ as an $\cO{\X}$-module.  Hence the right $\sD{\X}$-module $\cMpi[n-1]$ is generated by operators of the form $P_f : g \mapsto g\hook{\dd f}\mu = -[f,g\mu]$.  The differential $\dpi P_f \in \cMpi[n]$ therefore acts on $g$ by the formula
\begin{align*}
\dpi P_f g &= -[\pi,[f,g\mu]] \\
&= -[[\pi,f],g\mu] \\
&= \lie{\Ham{f}}(g\mu) \\
&= (\lie{\Ham{f}}g - g\Delta{\Ham{f}})\mu \\
&= (\lie{\Ham{f}}g + \Mod(f) \cdot g)\mu
\end{align*}
where we have used the Jacobi identity for $[-,-]$ and the identities \eqref{eqn:div} and \eqref{eqn:div-ham}.  We conclude that under the isomorphism $\cM[n] \cong \sD{\X}$, the differential $\dpi P_f \in \cM[n]$ corresponds to the generator $\Mod(f) + \Ham{f}$ of the ideal $\cI \subset \cO{\X}$, as desired.
\end{proof}

\subsection{The modular foliation}

Most Poisson manifolds are not Calabi--Yau, so it will typically be impossible to define a global modular vector field.  But any manifold admits covolumes on small enough open sets, and the modular vector fields defined in these local patches can be glued to produce a canonical foliation of $\X$, as follows.

Suppose that $\mu$ and $\mu'$ are covolumes defined on some simply connected open subset of $\X$,
so that $\mu' = g \mu$ for some nonvanishing function $g \in \cO{\X}$.  Then one easily calculates that their modular vector fields differ by a Hamiltonian:
\[
\Mod' = \Mod - \Ham{\log g},
\]
where $\log g$ denotes any branch of the  logarithm of $g$. Hence, the $\cO{\X}$-submodule
\[
\sF = \abrac{\Mod,\Ham{f} \,|\, f \in \cO{\X}} \subset \cT{\X}
\]
is independent of the choice of covolume used to define $\Mod$, and it therefore makes sense globally, as a coherent subsheaf of $\cT{\X}$.  Moreover, $\sF$ is involutive for the Lie bracket, as follows from the $\Mod$-invariance of $\pi$:
\[
[\Mod,\Ham{f}] = \lie{\Mod}(\hook{\dd f}\pi) = \hook{\dd \Mod(f)} \pi = \Ham{\Mod(f)}
\]
See also~\cite[Remark 7]{Gualtieri2013a} for a more global description of $\sF$ using Lie bialgebroids.

The involutive subsheaf $\sF\subset \cT{\X}$ defines a canonical foliation of $\X$, which we call the \defn{modular foliation of $(\X,\pi)$.} Its leaves evidently come in two types:
\begin{itemize}
\item The even-dimensional leaves are described locally as the symplectic leaves to which the modular vector field is tangent.
\item The odd-dimensional leaves are described locally by taking a symplectic leaf $\bL$ to which the modular vector field $\Mod$ is transverse, and considering the orbit of $\bL$ under the flow of $\Mod$.
\end{itemize} 

\begin{example}\label{ex:surf-modfol}
Using  \autoref{ex:surf-intro} and \autoref{ex:surf-CY}, we can describe the modular foliation for an arbitrary Poisson surface $(\X,\pi)$.  Let $\Y\subset \X$ be the zero locus.  Then there is a single two-dimensional leaf, namely the open set $\U = \X\setminus \Y$.  The one-dimensional leaves are the connected components of the smooth locus $\Y \setminus \Y_\sing$.  Finally, the zero-dimensional leaves are the points of the singular locus $\Y_\sing \subset\Y$.  Here $\Y_\sing$ is taken in the scheme-theoretic sense, i.e.~it includes all irreducible components of $\Y$ that occur with multiplicity greater than one.  \qed
\end{example}

We now prove our main holonomicity condition:

\begin{theorem}\label{thm:modular}
\thmmod
\end{theorem}

\begin{proof} Both problems are local so we may assume that $\X$ is Calabi--Yau with modular vector field $\Mod$.  For the first statement, it is enough to show that the local generators of $\sF$ annihilate $\Char{\X,\pi}$.  Since the Hamiltonian vector fields were dealt with in \autoref{lem:hams}, it remains to show that $\Mod$ also annihilates $\Char{\X,\pi}$.  But the derivation rule \eqref{eqn:bv-der} gives the homotopy formula
\begin{align*}
\BV{}\dpi + \dpi \BV{} &= \BV{}[\pi,-] + [\pi,\BV{}-] \\
&= [\BV{}\pi,-] \\
&= \lie{\Mod}.
\end{align*}
and hence the result follows by the same argument as in \autoref{lem:hams}.  

For the second statement, suppose that $\bL \subset \X$ is an even-dimensional modular leaf (a symplectic leaf to which $\Mod$ is tangent).   Using Weinstein's splitting theorem as in \autoref{sec:nondegen}, we may reduce to the case in which $\bL = \{p\}$ is a point where both $\pi$ and $\Mod$ vanish.  To see that the characteristic variety contains the conormal space $\ctb_p{\X}$, it is enough to show that the module $\scoH[n]{\cMpi}=\cMpi[n]/\dpi\cMpi[n-1]$ has a nontrivial quotient supported at $p$.  Indeed, suppose that $\m \subset \cO{\X}$ is the ideal of functions vanishing at $p$. Then we have $\Mod(f) \in \m$ and $\Ham{f} \in \m \cT{\X}$ for any $f \in \cO{\X}$.  Hence $\Mod(f)+\Ham{f} \in \m\sD{\X}$, so that the ideal $\cI \subset \sD{\X}$ of \autoref{prop:pres} is contained in $\m\sD{\X}$, and we have a surjection
\[
\cMpi[n]/\dpi \cMpi[n-1] \cong \sD{\X}/\cI \to \sD{\X}/\m\sD{\X}
\]
But $\sD{\X}/\m\sD{\X}$ is the ``delta-function'' D-module at $p$; in particular, it is nonzero.   We conclude that the conormal bundle to $\bL$ lies in $\Char{\X,\pi}$. 
\end{proof}

\begin{remark}
The second statement in \autoref{thm:modular} can be strengthened: one can show that any closed even-dimensional modular leaf $\bL \subset \X$ supports a canonical nontrivial quotient of $\cMpi[n]/\dpi \cMpi[n-1]$.  It corresponds to the local system on $\bL$ given by the top-degree cohomology of the transverse  Poisson structure. \qed
\end{remark}

\subsection{Holonomicity and log symplectic forms}
\label{sec:nondegen}

Suppose that $(\X,\pi)$ is a Poisson manifold whose dimension $n = \dim \X$ is even.  Then $\pi$ has an open dense symplectic leaf $\U \subset \X$ if and only if the Pfaffian $\pi^{n/2} \in \Gamma(\X,\det{\cT{\X}})$ is a nonzero section of the anticanonical line bundle.  If $\U \ne \X$, its complement is the \defn{degeneracy hypersurface} $\Y\subset \X$ defined by the vanishing of $\pi^{n/2}$.  We say that $(\X,\pi)$ is \defn{log symplectic} if $\Y$ is reduced.  (See \cite{Goto2002} for the original definition in terms of logarithmic differential forms and \cite[Lemma 2.6]{Pym2017} for the equivalence with definition given here.)

Applying  \autoref{thm:modular} to the description of the modular foliation of surfaces in \autoref{ex:surf-modfol}, we obtain a geometric characterization of holonomic surfaces:
\begin{corollary}\label{cor:surf-red}
A Poisson surface is holonomic if and only if it is log symplectic.
\end{corollary}

In higher dimension, the relation between these notions is weaker:
\begin{proposition}\label{p:hol-ls}
Every holonomic Poisson manifold is log symplectic.  Conversely, every log symplectic manifold is holonomic away from the singular locus of its degeneracy hypersurface.
\end{proposition}

\begin{proof}
For the first statement, it is enough to check reducedness of the degeneracy hypersurface $\Y$ on the open dense subset $\Y_0 \subset \Y$ where $\pi$ has locally constant rank.  By Weinstein's splitting theorem, the transverse Poisson structure $(\X',\pi')$ along $\Y_0$ must vanish on a hypersurface $\Y'\subset \X'$, and by \autoref{thm:modular} the modular vector field cannot vanish identically on $\Y'$. Hence by \autoref{lem:vanish-surf} below, the transverse Poisson structure is two-dimensional.  The statement therefore follows from the surface case (\autoref{cor:surf-red} above).

For the second statement, note that at any smooth point of $\Y$, the symplectic leaf must have codimension two, so that the transverse structure is two-dimensional and the statement follows from the surface case once again.
\end{proof}

\begin{lemma}\label{lem:vanish-surf}
Suppose that a Poisson structure $\pi$ vanishes identically on a hypersurface $\Y \subset \X$, where its modular vector field is locally nonvanishing.  Then $\pi$ generically has rank two.
\end{lemma}

\begin{proof}
The problem is local, so we may assume that $\X$ is Calabi-Yau with modular vector field $\Mod$.  Let $p \in \Y_\red$ be a smooth point of the underlying reduced hypersurface $\Y_\red \subset \Y$ and assume that $\Mod$ is nonzero at $p$.

Let $y$ be a defining equation for $\Y_\red$ near $p$, and let $w$ be a function such that $(\lie{\Mod}w)(p) \ne 0$.  Since the Hamiltonian vector field $\Ham{w}$ vanishes on $\Y$ we have $\Ham{w} = y \eta$ for some vector field $\eta$.  Moreover
\[
0 \ne (\lie{\Mod}w)(p) = (-\BV{}\Ham{w})(p) = - (y\BV{}\eta)(p) +(\lie{\eta}y)(p),
\]
so that $(\lie{\eta}y)(p) \ne 0$.  Hence $\eta$ generates a nonsingular foliation by curves transverse to the hypersurface $\Y$.  We may therefore complete the pair of functions $(w,y)$ to a coordinate system $(w,y,x_1,\ldots,x_{n-2})$ such that $\lie{\eta}x_i = 0$ for all $i$.  Put differently, we have $\Ham{w} = g\cvf{y}$ for some function $g$.  We then have
\[
\lie{\Ham{w}}\{x_i,x_j\} = \{\lie{\Ham{w}}x_i,x_j\} + \{x_i,\lie{\Ham{w}}x_j\} = 0
\]
for all $i$ and $j$, and hence $\{x_i,x_j\}$ must be independent of $y$.  But $\{x_i,x_j\}$ vanishes when $y=0$, so we must have $\{x_i,x_j\} = 0$ identically.  Hence the only nontrivial Poisson brackets of the coordinate functions are the ones involving $y$, i.e.~we have $\pi = \cvf{y} \wedge \Ham{y}$, which generically has rank two.
\end{proof}

In the singular locus of $\Y$, reducedness is no longer a sufficient condition for holonomicity.  For example, take $\X = \CC^n$ with coordinates $x_1,\ldots,x_n$ and consider a Poisson structure of the form
\[
\pi = \sum_{i ,j} \lambda_{ij} (x_i\cvf{x_i}) \wedge (x_j\cvf{x_j})
\]
where $\lambda = (\lambda_{ij})$ is a nondegenerate skew-symmetric matrix of constants.  For any such matrix, the degeneracy divisor $\Y$ is the union of the coordinate hyperplanes $\{x_i = 0\}$, a reduced normal crossings divisor.  But holonomicity puts a further genericity condition on $\lambda$. For example the ``general position'' and ``P-normal'' conditions of \cite{Ran2017a,Ran2017b} will guarantee holonomicity.  On the other hand \cite[Example 3.7]{Caine2017} gives an example for $\X=\CC^4$ where the Lichnerowicz cohomology is infinite-dimensional; the failure of holonomicity can be explained by the fact that the modular foliation has infinitely many zero-dimensional leaves.


\section{Perverse sheaves and applications}
\label{sec:compfact}

\subsection{Perversity of the Lichnerowicz complex}
\label{sec:perv}

For a general Poisson manifold, the full complex $\cMpi$ of $\sD{\X}$-modules will be required in order to encode the Lichnerowicz complex.  But in the holonomic case it reduces to its top cohomology module, which we described in \autoref{prop:pres}:
\begin{proposition}\label{prop:acyclic}
If $(\X,\pi)$ is a holonomic Poisson manifold of dimension $n$, then $\cMpi$ is exact in degrees $0,\ldots,n-1$, so that we have a quasi-isomorphism
\[
\cMpi{}[n] \cong \scoH[n]{\cMpi} = \cMpi[n]/\dpi \cMpi[n-1].
\]
Here $[n]$ denotes a shift in degree by $n$, so that $(\cMpi{}[n])^j = \cMpi[j+n]$, and the right hand side is a single module, viewed as a complex concentrated in degree zero.
\end{proposition}

\begin{proof}
Let $k \ge 0$ be the smallest integer such that $\cMpi$ is exact in degrees less than $k$, but the cohomology in degree $k$ is nonzero.  We must show that $k = n$.

To begin, we observe that the complex
\[
\cMpi[\le k] = (\xymatrix{ \cMpi[0] \ar[r] & \cMpi[1]\ar[r] &  \cdots \ar[r] & \cMpi[k]})
\]
is a projective resolution of the module $\sN = \cMpi[k]/\dpi \cMpi[k-1]$.  In other words, we have a quasi-isomorphism $\sN \cong \cMpi[\le k][k]$.

Let $\sH = \scoH[k]{\cMpi}$ be the $k$th cohomology module.  Evidently, we have an inclusion $\sH \subset \sN$.  In particular, $\sHom{\sH,\sN} \ne 0$.  On the other hand, we have a spectral sequence whose first page is
\[
\sExt[i]{\sH,\cMpi[j+k]} \implies \sExt[i+j]{\sH,\sN}.
\]
Since the modules $\cMpi$ are projective, a theorem of Roos~\cite{Roos1972}  implies the vanishing $\sExt[i]{\sH,\cMpi[j+k]} = 0$ whenever $i < \codim \Char{\sH}$ (see also~\cite[Theorem D.4.3]{Hotta2008}).  This module also vanishes when $j < -k$, since  $\cMpi$ is concentrated in nonnegative degrees.  It follows that $\sExt[l]{\sH,\sN} = 0$ whenever $l < \codim \Char{\sH}-k$.  Since $\sHom{\sH,\sN} = \sExt[0]{\sH,\sN} \ne 0$, we conclude that $\codim \Char{\sH} \le k$.  But the complex is holonomic, so $\codim \Char{\sH} = n$.  Therefore $k = n$, as desired.
\end{proof}

\begin{remark}
The proof shows more generally that any complex of locally free $\sD{\X}$-modules of the form $\cM = (\cM[0]\to\cdots \to \cM[\dim \X])$ with holonomic cohomology is quasi-isomorphic to its top cohomology.  This is a noncommutative analogue of the ``acyclicity lemma'' in algebraic geometry, e.g.~\cite[Lemma 3]{Buchsbaum1973}. \qed
\end{remark}

This simplification has an important consequence for the Lichnerowicz complex.  Recall that given any complex $\cM$ of right $\sD{\X}$-modules, we may form its ``de Rham complex'' by taking the derived tensor product $\cM \otimes_{\sD{\X}}^\mathbb{L} \cO{\X}$.  The result only depends on $\cM$ up to quasi-isomorphism.  A deep theorem in analysis, due to Kashiwara~\cite{Kashiwara1974/75}, implies that when $\cM$ is a single holonomic $\sD{\X}$-module concentrated in degree zero, its de Rham complex is a \defn{perverse sheaf}.  Combined with \autoref{prop:pres}, it gives the following

\begin{corollary}\label{thm:perverse}
If $(\X,\pi)$ is a holonomic Poisson manifold, then the shifted Lichnerowicz complex $(\Tpol{\X}[n],\dpi)$ is a perverse sheaf.
\end{corollary}

We recall that perverse sheaves are, in particular, complexes of sheaves with finite-dimensional stalks that are locally constant along the strata of a Whitney stratification of $\X$.  Moreover, there are tight constraints on  the supports of the cohomology sheaves in terms of their degrees.  Perverse sheaves form an \emph{abelian} category whose simple objects  are intersection cohomology complexes of local systems supported on subvarieties of $\X$, and in general any object in the category can be built as iterated extensions of such simple objects (called its ``composition factors'').  We refer the reader to the standard texts, such as \cite{Beilinson1982,Cataldo2009,Hotta2008,Kashiwara1994}, for a thorough exposition.  For the remainder of the paper, we give some examples of how perversity can be applied to calculate the Lichnerowicz cohomology in low degrees, leaving a more detailed analysis for future work.

Let $\U = \X\setminus \Y$ be the open symplectic leaf and let $j : \U \to \X$ be the inclusion.  The map $\pis : \forms[1]{\X} \to \cT{\X}$ given by contraction with $\pi$ is an isomorphism over $\U$ and induces an isomorphism $\Tpol{\X}|_\U \cong \forms{\U}$ with the de Rham complex. By adjunction we obtain a map $\Tpol{\X} \to j_*\forms{\U}$.  It is also convenient to consider the logarithmic de Rham complex $\forms{\X}(\log \Y) \subset j_*\forms{\U}$ of P.~Deligne~\cite{Deligne1970} and K.~Saito~\cite{Saito1980}, which consists of forms $\omega$ such that $\omega$ and $\dd \omega$ have, at worst, poles of order one along $\Y$. 

 Because $(\X,\pi)$ is log symplectic, the map $\pis : \forms[1]{\X} \to \cT{\X}$ extends to an isomorphism $\forms[1]{\X}(\log\Y) \cong \cT{\X}(-\log\Y)$, where $\cT{\X}(-\log \Y)\subset \cT{\X}$ is the sheaf of vector fields tangent to $\Y$.  We thus obtain a commutative diagram of complexes
\[
\xymatrix{
& \forms{\X}(\log \Y) \ar[rd]\ar[ld] \\
\Tpol{\X} \ar[rr] && j_*\forms{\U}.
}
\]
It is known that both diagonal maps are quasi-isomorphisms over the open set $\X \setminus \Y_\sing$, where $\Y_\sing$ is the singular locus of $\Y$.  For $j_*\forms{\U}$, this is standard, e.g.~\cite{Deligne1970,Hodge1955}.  For $\Tpol{\X}$, this can be found, for example, in \cite[Lemma 3]{Marcut2014}.  Alternatively, one can show that the corresponding D-modules are isomorphic: at any point of $\Y \setminus \Y_\sing$, Weinstein's splitting theorem~\cite{Weinstein1983} reduces the problem to the case of Poisson surfaces, where it follows by a direct calculation using \autoref{prop:pres}, and the fact that all three D-modules have length two.

As a result, we obtain an exact sequence of perverse sheaves
\begin{align}
\xymatrix{
0 \ar[r] &  \sK \ar[r]& \Tpol{\X}[n] \ar[r]& j_*\forms{\U}[n] \ar[r] & \sK' \ar[r] &  0 
}\label{eqn:tpol-dr}
\end{align}
where the kernel and cokernel $\sK$ and $\sK'$ are supported on $\Y_\sing$.  For a given holonomic Poisson manifold, the explicit determination of the perverse sheaves $\sK$ and $\sK'$ may be quite difficult; they will be built from many composition factors given by intersection cohomology complexes of local systems on singular strata of the Poisson structure.  But perversity gives us a key piece of information: the hypercohomology groups of $\sK$ and $\sK'$ must vanish outside the range $[-d,d]$ where $d = \dim \Y_\sing$.

\subsection{Normal degeneracy divisors}

The simplest case is the one in which the degeneracy divisor $\Y$ is normal,   i.e.~the codimension of $\Y_\sing$ in $\X$ is at least three.  We recall from \cite{Gualtieri2013a} that if $\Y_\sing$ is nonempty, it must have codimension \emph{at most} three.  Splitting \eqref{eqn:tpol-dr} into a pair of short exact sequences of perverse sheaves, and using the vanishing of the hypercohomology $\coH[j]{\X,\sK}$ and $\coH[j]{\X,\sK'}$ for $j<3-n$ we obtain
\begin{theorem}\label{thm:normal}
\thmnormal
\end{theorem}

The calculation of the cohomology is facilitated by the following
\begin{lemma}\label{lem:IH}
If $\Y \subset \X$ is a normal hypersurface with complement $\U = \X\setminus \Y$, then there is an exact sequence
\begin{align}\begin{gathered}
\xymatrix{
0\ar[r] & \coH[1]{\X;\CC} \ar[r] & \coH[1]{\U;\CC}\ar[r] & \coH[0]{\Y;\CC} \ar[lld] \\ & \coH[2]{\X;\CC} \ar[r] & \coH[2]{\U;\CC} \ar[r] & \IH[1]{\Y;\CC}
}\end{gathered}\label{eqn:gysin}
\end{align}
where $\IH{\Y;\CC}$ denotes the intersection cohomology of $\Y$ and the connecting homomorphism $\coH[0]{\Y;\CC} \to \coH[2]{\X;\CC}$ is the Gysin map.
\end{lemma}

\begin{proof}
Because $\Y$ is a hypersurface, $Rj_*\CC_\U[n]$ is perverse, so the standard attaching triangle associated with the closed embedding $i : \Y \to \X$ and the open complement $j : \U \to \X$ gives an exact sequence of perverse sheaves
\begin{align}
\begin{gathered}
\xymatrix{
0 \ar[r] & \CC_{\X}[n] \ar[r] & Rj_*\CC_\U[n] \ar[r] & \sQ \ar[r] & 0,
}\end{gathered} \label{eqn:ses1}
\end{align}
where $\sQ = Ri_*i^!\CC_\X[n+1]$ and $i^!$ denotes the exceptional inverse image. 

On the smooth locus $\Y_\reg = \Y\setminus \Y_\sing$, we have  $\sQ|_{\Y_\reg} \cong \CC_{\Y_\reg}[n-1]$, so the intersection complex $\sIC{\Y}[n-1]$ is the unique composition factor of $\sQ$ whose support has codimension one, i.e.~we have sub-perverse sheaves $\sQ_0 \subset \sQ_1 \subset \sQ$ such that $\sQ_1/\sQ_0 \cong \sIC{\Y}[n-1]$ and $\sQ_0$ and $\sQ/\sQ_1$ are supported on $\Y_\sing$, giving exact sequences
\begin{align}
\begin{gathered}
\xymatrix{
0 \ar[r] & \sQ_0 \ar[r] & \sQ_1 \ar[r] & \sIC{\Y}[n-1] \ar[r] & 0
}\end{gathered} \label{eqn:ses2}
\end{align}
\vspace{-2em}
\begin{align}
\begin{gathered}
\xymatrix{
0 \ar[r] & \sQ_1 \ar[r] & \sQ \ar[r] & \sQ/\sQ_1 \ar[r] & 0.
}\end{gathered} \label{eqn:ses3}
\end{align}
 Since $\dim \Y_{\sing} \le n-3$, we have $\coH[j]{\sQ_0}=\coH[j]{\sQ/\sQ_1} = 0$ for $j < -(n-3)$.  The desired exact sequence \eqref{eqn:gysin} now follows from the long exact sequences associated to \eqref{eqn:ses1}, \eqref{eqn:ses2} and \eqref{eqn:ses3}.  
 The statement about the connecting homomorphism is standard when $\Y$ is smooth (e.g.~\cite[Proposition 3.2.11]{Dimca2004}).  The case when $\Y$ is normal follows: the restriction maps $\coH[0]{\Y;\CC} \to \coH[0]{\Y\setminus\Y_\sing;\CC}$ and $\coH[2]{\X;\CC}\to\coH[2]{\X\setminus \Y_\sing;\CC}$ are isomorphisms because $\codim(\Y_\sing,\X)\ge 3$.
\end{proof}

\subsection{Elliptic Poisson structures}
 As an example of the case when $\Y$ is normal, we consider Poisson structures associated with smooth curves of genus one (elliptic curves without an origin).

As a warmup, note that a Poisson structure $\pi$ on the projective plane $\PP^2$ will vanish on an anticanonical divisor $\E \subset \PP^2$, i.e.~a cubic curve.  If $\E$ is smooth, it will have genus one. Moreover, the modular vector field, while not globally well-defined on $\PP^2$, will restrict to a canonical nonvanishing vector field $\Mod \in \Gamma(\E,\cT{\E})$.  Note that the embedding $\E \subset \PP^2$ is unique up to projective transformations, and  determines $\pi$ up to rescaling.  Hence, the pair $(\E,\Mod)$ determines $\pi$ up to isomorphism.  Correspondingly, the deformation space $\Hpi[2]{\PP^2} \cong \coH[2]{\U;\CC}$ has dimension two: it is isomorphic to $\coH[1]{\E;\CC}$ by the Gysin sequence.

Feigin and Odesskii \cite{Feigin1989,Feigin1998} constructed higher-dimensional analogues of these Poisson structures. Given a pair $(d,r)$ of coprime integers, a smooth projective curve $\E$ of genus one, and a vector field $\Mod$ on $\E$, they produce a Poisson structure $\pi = q_{d,r}(\E,\Mod)$ on the projective space $\PP^n$ where $n=d-1$.

We shall restrict to the simplest case, in which $n$ is even and $r=1$.  In this case, $\pi$ is log symplectic and the classical geometry of its modular foliation is understood~\cite[Section 8]{Gualtieri2013a}.  The  zero locus of $\pi$ is a copy of $\E$, embedded in $\PP^n$ as an elliptic normal curve, and the vector field $\Mod$ on $\E$ is exactly the restriction of the modular vector field.  There is a single even-dimensional modular leaf, namely the open symplectic leaf $\U$.  Meanwhile, there is a unique modular leaf of each odd dimension $1,3,\ldots,n-1$, given by an appropriate secant variety of $\E$.  Hence, by \autoref{thm:modular}, $\pi$ is holonomic.

Let $k = \tfrac{n}{2}-1$. Then the degeneracy hypersurface $\Y \subset \PP^n$ is the union of all the secant $k$-planes of $\E$, i.e.~the $k$-planes in $\PP^n$ that meet $\E \subset \PP^n$ in $k+1$ points, counted with multiplicity. It is normal, so we may apply \autoref{thm:normal} and \autoref{lem:IH} to compute $\Hpi{\PP^n}$ in low degrees.

Now $\Y$ is connected and $[\Y] \in \coH[2]{\PP^n;\CC}$ is a generator, so the Gysin map $\coH[0]{\Y;\CC} \to \coH[2]{\PP^n;\CC}$ is an isomorphism.  Meanwhile $\coH[1]{\PP^n;\CC} = 0$, giving
\begin{align*}
\coH[1]{\U;\CC} &= 0 & \coH[2]{\U;\CC}&\subset\IH[1]{\Y;\CC}
\end{align*}
by \autoref{lem:IH}.  By the decomposition theorem~\cite{Beilinson1982}, the intersection cohomology embeds  as a subspace $\IH[1]{\Y;\CC}\subset \coH[1]{\Y';\CC}$ where $\Y' \to \Y$ is any resolution of singularities.  In the case at hand, there is a canonical resolution described in detail in \cite[Proposition 8.15]{GrafvBothmer2004}, namely
\[
\Y' = \set{(\D,p) \in \bS^{k+1}(\E) \times \PP^n}{p \in \mathrm{span}(\D)}
\]
where $\bS^{k+1}(\E) = (\E \times \cdots \times \E)/\bS_{k+1}$ is the symmetric power of $\E$, and $\mathrm{span}(\D)$ is the unique $k$-plane that contains the degree-$(k+1)$ divisor $\D \subset \E$.  That such a $k$-plane is unique is the content of \cite[Lemma 13.2]{GrafvBothmer2004}.  The resolution map $\Y' \to \Y$ is simply the projection $(\D,p)\mapsto p$. 

Now the fibre of $\Y'$ over  $\D \in \bS^{k+1}(\E)$ is the $k$-plane $\mathrm{span}(\D)$, so  the projection $\Y' \to \bS^{k+1}(\E)$ is a $\PP^k$-bundle.  In turn, the map $\bS^{k+1}(\E) \to \Pic[k+1]{\E} \cong \E$, sending a divisor to its linear equivalence class, is also a $\PP^k$-bundle.  Hence $\Y'$ is a $\PP^k$-bundle over a $\PP^k$-bundle over $\E$, so the projective bundle formula gives $\coH[1]{\Y';\CC} \cong \coH[1]{\E;\CC}$.  We therefore have a canonical sequence of inclusions
\begin{align*}
\Hpi[2]{\PP^n} \subset \coH[2]{\U;\CC}\subset \IH[1]{\Y;\CC}\subset \coH[1]{\Y';\CC} = \coH[1]{\E;\CC}.
\end{align*}

But the Feigin--Odesskii Poisson structures give a two-parameter family of nontrivial deformations of $\pi$, so $\dim \Hpi[2]{\PP^n} \ge 2$ by \autoref{thm:gc-def}.  On the other hand, $\dim \coH[1]{\E;\CC} =2$.  Hence all of these inclusions are isomorphisms, giving
\begin{theorem}\label{thm:elliptic}
\thmelliptic
\end{theorem}

\subsection{Surfaces}
\label{sec:surf}
To illustrate how complications can arise when the divisor $\Y\subset \X$ is not normal, we sketch the case of a holonomic Poisson surface $(\X,\pi)$, following  Goto~\cite{Goto2016}.  We recall that for surfaces, holonomicity is equivalent to the log symplectic condition (\autoref{cor:surf-red}), so we adopt this terminology.  Furthermore, we note that for surfaces, the sheaf $\forms[1]{\X}(\log \Y)$ is always locally free and the complex $\forms{\X}(\log \D)[2]$ is always perverse; see \cite[(1.7)]{Saito1980} and \cite[Corollary 4.2.2]{CalderonMoreno1999}.

The image of the map $\forms[1]{\X}(\log\Y) \to \cT{\X}$ is the subsheaf $\cT{\X}(-\log\Y)$ of vector fields tangent to $\Y$, while the image of $\forms[2]{\X}(\log \Y) \to \Tpol[2]{\X}$ is the sheaf of bivectors that vanish on $\Y$.  Hence $\sQ = \coker(\forms{\X}(\log\D)\to\Tpol{\X})$ has the form
\[
\sQ = (\xymatrix{0 \ar[r] &
\cT{\X}/\cT{\X}(-\log\Y) \ar[r] & \Tpol[2]{\X}|_\Y
}).
\]
To understand the meaning of $\sQ$, we observe that since $\Y$ is the zero locus of $\pi \in \Tpol[2]{\X}$, the tangent complex of $\Y$ has the form
\[
\Tcx{\Y} = \rbrac{\xymatrix{
 \cT{\X}|_\Y \ar[r]^-{\dd \pi|_\Y} & \Tpol[2]{\X}|_{\Y}
}},
\]
where $\dd \pi$ is the derivative of $\pi$ along $\Y$.  The zeroth cohomology $\scoH[0]{\Tcx{\Y}}$ is the usual sheaf of vector fields  on $\Y$, or equivalently the image of the restriction map $\cT{\X}(-\log \Y) \to \cT{\X}|_\Y$.  Meanwhile, the first cohomology $\scoH[1]{\Tcx{\Y}} \cong \Tpol[2]{\X}|_{\Y_\sing}$ is the space of infinitesimal deformations of the singularities of $\Y$; its support is the singular locus $\Y_\sing$, and the dimension of its stalk at $p \in \Y_\sing$ is the Tjurina number, i.e.~the dimension of the Jacobi ring $\cO{\X,p}/(f,\cvf{w} f,\cvf{z}f)$ where $f$ is a local defining equation for $\Y$ in coordinates $w,z$.

It follows easily that $\scoH[j]{\sQ} = 0$ for $j \ne 2$ and $\scoH[2]{\sQ} \cong \scoH[1]{\Tcx{\Y}}$.  Hence we have an exact sequence of perverse sheaves:
\begin{align}
\xymatrix{
0 \ar[r] & \forms{\X}(\log \Y)[2] \ar[r] & \Tpol{\X}[2] \ar[r] & \scoH[1]{\Tcx{\Y}} \ar[r] & 0
} \label{eqn:goto}
\end{align}

We now note that since $\X$ is a surface, the map $\forms{\X}(\log \Y) \to j_*\forms{\U}$ is a quasi-isomorphism if and only if $\Y$ can be written locally as the zero set of a quasi-homogeneous polynomial~\cite{CalderonMoreno2002,Castro-Jimenez1996}.  This condition always holds when $(\X,\pi)$ is a projective log symplectic surface, as follows from the classification~\cite[Section 7]{Ingalls1998} or by examining what happens when we blow up the possible minimal models.  In this case, the map $\Tpol{\X} \to j_*\forms{\U}$ splits the sequence~\eqref{eqn:goto}, giving
\[
\Tpol{\X}[2] \cong j_*\forms{\U}[2]\oplus \scoH[1]{\Tcx{\Y}},
\]
and identifying the contributions $\sK \cong \scoH[1]{\Tcx{\Y}}$ and $\sK'=0$ in \eqref{eqn:tpol-dr}.  We arrive at the following result, first stated by Goto~\cite[Proposition 6]{Goto2016} for nodal curves:

\begin{theorem}\label{thm:surf}
If the degeneracy curve $\Y \subset \X$ of a log symplectic surface $(\X,\pi)$ has only quasi-homogeneous singularities (e.g.~if $\X$ is projective) then we have
\[
\Hpi[k]{\X} \cong \begin{cases}
\coH[k]{\U;\CC}  & k \ne 2 \\
\coH[2]{\U;\CC}\oplus \bigoplus_{p \in \Y_\sing} \scoH[1]{\Tcx{\Y}}_p  & k = 2
\end{cases}
\]
where $\U = \X\setminus \Y$ is the symplectic locus.
\end{theorem}
This result recovers various calculations of Poisson cohomology for surfaces that were obtained by more direct means~\cite{Hong2011,Monnier2002}, and matches the recent calculation~\cite{Belmans2017} of the Hochschild cohomology of the quantizations in the case of $\PP^2$ and $\PP^1\times \PP^1$, as expected from Kontsevich's formality theorem~\cite{Kontsevich2001}.

\bibliographystyle{hyperamsplain}
\bibliography{holonomic-poisson}

\end{document}